\documentclass[11pt]{article}
\usepackage{amsfonts}
\usepackage{amsmath,amssymb}
\usepackage{amsthm}
\usepackage[mathscr]{euscript}
\usepackage{mathrsfs}
\usepackage{indentfirst}
\usepackage{color}
\usepackage{accents}
\usepackage{enumerate}
\usepackage{lineno}

\newtheorem{theorem}{Theorem}[section]
\newtheorem{lemma}{Lemma}[section]

\newtheorem{remark}{Remark}[section]
\newtheorem{proposition}{Proposition}[section]
\numberwithin{equation}{section}

\allowdisplaybreaks

\newcommand{\be}{\begin{equation}}
\newcommand{\ee}{\end{equation}}
\newcommand\bes{\begin{eqnarray}} \newcommand\ees{\end{eqnarray}}
\newcommand{\bess}{\begin{eqnarray*}}
\newcommand{\eess}{\end{eqnarray*}}
\newcommand{\bbbb}{\left\{\begin{aligned}}
\newcommand{\nnnn}{\end{aligned}\right.}
\newcommand{\bea}{\begin{align*}}
\newcommand{\eea}{\end{align*}}

\newcommand\ep{\varepsilon}

\newcommand\kk{\left}
\newcommand\rr{\right}
\newcommand\dd{\displaystyle}

\newcommand\dx{{\rm d}x}
\newcommand\dy{{\rm d}y}

\newcommand\yy{\infty}

\newcommand\ud{\underline}

\begin{document}\thispagestyle{empty}
\setlength{\baselineskip}{16pt}

\begin{center}
 {\LARGE\bf Longtime behaviors of an epidemic model}\\[2mm]
 {\LARGE\bf with nonlocal diffusions and a free boundary: rate of accelerated spreading\footnote{This work was supported by NSFC Grants
 12301247, 12171120.}}\\[4mm]
{\Large Lei Li}\\[0.5mm]
{School of Mathematics and Statistics, Henan University of Technology}\\
{Zhengzhou, 450001, China}\\
{\Large Mingxin Wang}\footnote{Corresponding author. {\sl E-mail}:  mxwang@hpu.edu.cn}\\[1mm]
{\small School of Mathematics and Information Science, Henan Polytechnic University,}\\
{ Jiaozuo, 454003, China}
\end{center}

\date{\today}

\begin{quote}
\noindent{\bf Abstract.} This is the third part of our series of work devoted to the dynamics of an epidemic model with nonlocal diffusions and free boundary. This part is concerned with the rate of accelerated spreading for three types of kernel functions when spreading happens. By constructing the suitable upper and lower solutions, we get the rate of the accelerated spreading of free boundary, which is closely related to the behavior of kernel functions near infinity. Our results indicate that the heavier the tail of the kernel functions are, the faster the rate of accelerated spreading is. Moreover, more accurate spreading profiles for solution component $(u,v)$ are also obtained.


\end{quote}

\section{Introduction}\pagestyle{myheadings}
\renewcommand{\thethm}{\Alph{thm}}
{\setlength\arraycolsep{2pt}
In this paper, we continue  investigating the dynamics of the following epidemic model
  \bes\left\{\begin{aligned}\label{1.1}
&u_t=d_1\int_0^{h(t)}\!J_1(x-y)u(t,y)\dy-d_1u-au+H(v), \hspace{2mm} t>0, ~ x\in[0,h(t)),\\
&v_t=d_2\int_0^{h(t)}\!J_2(x-y)v(t,y)\dy-d_2v-bv+G(u), \hspace{3mm} t>0, ~ x\in[0,h(t)),\\
&u(t,h(t))=v(t,h(t))=0, \hspace{53mm} t>0,\\
&h'(t)=\int_0^{h(t)}\!\int_{h(t)}^{\yy}\big[\mu_1 J_1(x-y)u(t,x)
+\mu_2 J_2(x-y)v(t,x)\big]\dy\dx, \;\; t>0,\\
&h(0)=h_0>0,~ u(0,x)=u_0(x), ~ v(0,x)=v_0(x),\; ~ x\in[0,h_0],\\
 \end{aligned}\right.
 \ees
 where all parameters are positive constants. The kernel functions $J_1$ and $J_2$ satisfy
 \begin{enumerate}
\item[{\bf(J)}]$J\in C(\mathbb{R})\cap L^{\yy}(\mathbb{R})$, $J(x)\ge0$, $J(0)>0$, $J$ is even, $\dd\int_{\mathbb{R}}J(x)\dx=1$.
 \end{enumerate}
The initial value functions $u_0$ and $v_0$ are in $C([0,h_0])$, positive in $[0,h_0)$ and vanish at $x=h_0$.
 The nonlinear reaction terms $H$ and $G$ satisfy
 \begin{enumerate}
\item[{\bf(H)}]\; $H,G\in C^2([0,\yy))$, $H(0)=G(0)=0$, $H'(z),G'(z)>0$ in $[0,\yy)$, $H''(z), G''(z)<0$ in $(0,\yy)$, and $G(H(\hat z)/a)<b\hat{z}$ for some $\hat{z}>0$.
 \end{enumerate}

In this model, $u(t,x)$ and $v(t,x)$ stand for the spatial
concentrations of the bacteria and the infective human population, respectively, at time $t$ and location $x$ in the one dimensional habitat; $-au$ represents the natural death rate of the bacterial population and $H(v)$ denotes the contribution of the infective human to the growth rate of the bacteria; $-bv$ is the fatality rate of the infective human population and $G(u)$ is the infection rate of human population; $d_1$ and $d_2$, respectively, stand for the diffusion rates of bacteria and infective human; the spatial movements of $u$ and $v$ are approximated by the nonlocal diffusions.

The corresponding Cauchy problem of \eqref{1.1} with local diffusions
\bes\label{1.2}
\left\{\!\begin{aligned}
&u_{t}=d_1 u_{xx}-au+H(v), & &t>0,~x\in\mathbb{R},\\
&v_{t}=d_2 v_{xx}-bv+G(u), & &t>0,~x\in\mathbb{R}
\end{aligned}\right.
 \ees
 was first proposed  by Hsu and Yang \cite{HY} to model the spread of an oral-fecal transmitted epidemic.
Let
  \bess
  \mathcal{R}_0=\frac{H'(0)G'(0)}{ab}.
 \eess
It was proved in \cite{HY} that when $\mathcal{R}_0>1$, there exists a $c_*>0$ such that if and only if $c\ge c_*$, \eqref{1.2} has a positive monotone traveling wave solution that is unique up to translation. Moreover, the dynamics of the corresponding ODE system with positive initial value is governed by $\mathcal{R}_0$. When $\mathcal{R}_0<1$, $(0,0)$ is globally asymptotically stable; while when $\mathcal{R}_0>1$, there exists a unique positive equilibrium $(u^*,v^*)$ which is uniquely given by
  \bes\label{1.3}
  au^*=H(v^*), ~ ~ bv^*=G(u^*),\ees
and is globally asymptotically stable.

For the explanation of model \eqref{1.1}, one can refer to \cite{AMRT,CDLL,LLW}.  We in  \cite{LL1} proved that this model is well-posed, and its longtime behaviors are governed by the following spreading-vanishing dichotomy.
\begin{enumerate}[$(1)$]
\item \underline{Spreading:} necessarily $\mathcal{R}_0>1$, $ h_\yy:=\lim_{t\to\yy}h(t)=\yy$,
\[\lim_{t\to\yy}(u(t,x),v(t,x))=(U(x),V(x)) ~ ~ {\rm ~  in ~ }[C_{\rm loc}([0,\yy))]^2,\]
 where $(U(x),V(x))$ is the unique bounded positive solution of
 \bes\label{1.4}
\begin{cases}
 \dd d_1\int_{0}^{\yy}J_1(x-y)U(y)\dy-d_1U-aU+H(V)=0, ~ ~ x\in[0,\yy),\\
\dd d_2\int_{0}^{\yy}J_2(x-y)V(y)\dy-d_2V-bV+G(U)=0, ~ ~  x\in[0,\yy).
\end{cases}
\ees
\item \underline{Vanishing:} $h_{\yy}<\yy$ and $\lim_{t\to\yy}\|u(t,\cdot)+v(t,\cdot)\|_{C([0,h(t)])}=0$.
\end{enumerate}

It can be seen from \cite[Proposition 2.2]{LL1} that the unique bounded positive solution $(U,V)$ of \eqref{1.4} is continuous and strictly increasing for $x\ge0$, and  $(U(x),V(x))\to(u^*,v^*)$ as $x\to\yy$.

Later on, we in \cite{LL2} investigated the spreading speed when spreading occurs. We found that the spreading speed is determined by the following so-called semi-wave problem
\bes\left\{\begin{aligned}\label{1.5}
&d_1\int_{-\yy}^0J_1(x-y)\phi(y)\dy-d_1\phi+c\phi'-a\phi+H(\psi)=0, & & x\in(-\yy,0),\\
&d_2\int_{-\yy}^0J_2(x-y)\psi(y)\dy-d_2\psi+c\psi'-b\psi+G(\phi)=0, & & x\in(-\yy,0),\\
&\phi(-\yy)=u^*, ~ ~ \psi(-\yy)=v^*, ~ ~\phi(0)=\psi(0)=0,\\
&c=\int_{-\yy}^0\int_0^{\yy}\bigg(\mu_1J_1(x-y)\phi(x)+\mu_2J_2(x-y)\psi(x)\bigg)\dy\dx.
 \end{aligned}\right.
 \ees
 A rather complete understanding of such semi-wave problem has been obtained by Du and Ni \cite{DN1}. They showed that \eqref{1.5} has a unique solution triplet $(c_0,\phi_{c_0},\psi_{c_0})$ with $c_0>0$ and $(\phi_{c_0},\psi_{c_0})$ strictly decreasing in $(-\yy,0]$ if and only if $J_1$ and $J_2$ satisfy the following condition
 \begin{enumerate}
\item[{\bf(J1)}] $\dd\int_0^{\yy}xJ_i(x)\dx<\yy$\, for $i=1,2$.
 \end{enumerate}

Using the solution of \eqref{1.5} and some suitable comparison arguments, in \cite{LL2} we derived the spreading speed of \eqref{1.1}. That is,  if {\bf (J1)} holds, then
\bess
    \begin{cases}\dd\lim_{t\to\yy}\frac{h(t)}{t}=c_0, \\
    \dd\lim_{t\to\yy}\max_{[0,ct]}\big(|u(t,x)-U(x)|+|v(t,x)-V(x)|\big)=0 \; ~ {\rm for ~ any ~ } c\in[0,c_0),
    \end{cases}
    \eess
    where $c_0$ is unique speed of semi-wave problem \eqref{1.5}, and $(U,V)$ is the unique bounded positive solution of \eqref{1.4}.
  If {\bf (J1)} does not hold, then
    \bess
    \begin{cases}\dd\lim_{t\to\yy}\frac{h(t)}{t}=\yy, \\
    \dd\lim_{t\to\yy}\max_{[0,ct]}\big(|u(t,x)-U(x)|+|v(t,x)-V(x)|\big)=0 \; ~ {\rm for ~ any ~ } c\in[0,\yy).
    \end{cases}
    \eess
    The case $\lim_{t\to\yy}h(t)/t=\yy$ is usually called {\it accelerated spreading}.

    In this paper, we are interested in the rate of accelerated spreading for three types of kernel functions. For clarity, we need first to introduce the following two notations:
\begin{enumerate}[$(1)$]
  \item $s(t)\sim l(t)$ means that there exist positive constants $\varsigma_1$ and $\varsigma_2$ such that
  \[\varsigma_1l(|t|)\le s(|t|)\le \varsigma_2 l(|t|)~ {\rm ~  for ~ all ~ }|t|\gg1.\]
  \item $s(t)=o(l(t))$ means that $s(t)/l(t)\to0$ as $t\to\yy$.
\end{enumerate}

The three types of kernel functions we care about are
\bes
&\dd J_i\sim \frac{\ln^{\alpha}\ln|x|}{|x|\ln^{\beta} |x|} ~ {\rm for ~  }i=1,2, ~ \alpha\in\mathbb{R},{\rm   ~ and ~ }\beta>1,\label{1.6}\\
& \dd J_i\sim \frac{\ln^{\beta}|x|}{|x|^{\gamma}} {\rm ~ for ~ }i=1,2, ~ \gamma\in(1,2) ~ {\rm and ~ }\beta\in\mathbb{R},\label{1.7}\\
&\dd J_i\sim \frac{\ln^{\beta}|x|}{|x|^{2}}{\rm ~ for ~ }i=1,2,{\rm ~ and ~ }\beta\ge-1 \label{1.8}.
\ees

Clearly, as $|x|\to\yy$, the decaying rates for these three types of kernel functions satisfy
\[\eqref{1.8}\succ\eqref{1.7}\succ\eqref{1.6}.\]

We have the following main result.
 \begin{theorem}\label{t1.1}Let $(u,v,h)$ be the unique solution of \eqref{1.1} and spreading happen. Then the following statements are valid.
 \begin{enumerate}[$(1)$]
   \item If \eqref{1.6} holds, then
    \bess
     \begin{cases}
     \dd\ln h(t)\sim t^{\frac{1}{\beta}}\ln^{\frac{\alpha}{\beta}}t,\\
   \dd\lim_{t\to\yy}\max_{x\in[0,s(t)]}\big(|u(t,x)-U(x)|+|v(t,x)-V(x)|\big)=0 {\rm ~ for ~ any ~ }s(t){\rm ~ with ~ }\ln s(t)=o(t^{\frac{1}{\beta}}\ln^{\frac{\alpha}{\beta}}t).
   \end{cases}
   \eess
   \item If \eqref{1.7} holds, then
   \bess
    \begin{cases}
    \dd h(t)\sim t^{\frac{1}{\gamma-1}}\ln^{\frac{\beta}{\gamma-1}}t,\\
   \dd \lim_{t\to\yy}\max_{x\in[0,s(t)]}\big(|u(t,x)-U(x)|+|v(t,x)-V(x)|\big)=0 {\rm ~ for ~ any ~ }s(t)=o(t^{\frac{1}{\gamma-1}}\ln^{\frac{\beta}{\gamma-1}}t).
   \end{cases}
   \eess
   \item If \eqref{1.8} holds with $\beta>-1$, then
   \bess
   \begin{cases}
  \dd h(t)\sim t\ln^{\beta+1} t,\\
   \dd \lim_{t\to\yy}\max_{x\in[0,s(t)]}\big(|u(t,x)-U(x)|+|v(t,x)-V(x)|\big)=0 {\rm ~ for ~ any ~ }s(t)=o(t\ln^{\beta+1} t).
   \end{cases}
   \eess
   \item If \eqref{1.8} holds with $\beta=-1$, then
   \bess
     \begin{cases}
     \dd h(t)\sim t\ln \ln t,\\
   \dd\lim_{t\to\yy}\max_{x\in[0,s(t)]}\big(|u(t,x)-U(x)|+|v(t,x)-V(x)|\big)=0 {\rm ~ for ~ any ~ }s(t)=o(t\ln \ln t).
   \end{cases}
   \eess
 \end{enumerate}
 \end{theorem}

We would like to mention that some results like Theorem \ref{t1.1} for the Fisher-KPP equation in \cite{CDLL,LLW,DLZ} and the epidemic model in \cite{NV} can be proved by following analogous lines as in this paper. Our arguments are based on some suitably upper and lower solutions, and are mainly inspired by \cite{DN2,DN3,DN4}. In \cite{DN2}, Du and Ni studied the rate of accelerated spreading for the Fisher-KPP equation with nonlocal diffusion and free boundaries in one dimensional space for algebraic decay function, namely, those behaving like $|x|^{-\gamma}$ for $\gamma\in(1,2]$ near infinity. For the case in high dimensional space with  radially symmetry, Du and Ni \cite{DN3} obtained the rate of accelerated spreading when spreading happens. Moreover, the rate of accelerated spreading for an epidemic model with nonlocal diffusions and free boundaries was also discussed by Du, Ni and Wang \cite{DNW}. For the corresponding Cauchy problem with such nonlocal diffusion operator, accelerated spreading also can happen if kernel function converges to $0$ as $|x|\to\yy$ slower than any exponentially decaying function, nowadays well known as exponentially unbounded or heavy-tailed (see e.g. \cite{JGar,MaJe,BGHP,FKT,XLL}). On the other hand, with slowly decaying initial function, i.e., decaying slower than any exponentially decaying function, accelerated propagation can appear for the reaction-diffusion equation (see e.g. \cite{HR,Hen}), the fractional reaction-diffusion equation (see e.g. \cite{FYan,SV}), and the chemotaxis-growth system (see e.g. \cite{WX}). One will see that for reaction-diffusion equation, the initial function plays a similar role with the kernel function for nonlocal diffusion problem.

For the recent development on the free boundary problem with local or nonlocal diffusion from ecology, one can refer to, for example, \cite{DN5} for the approximation of local diffusion version in \cite{DL} by nonlocal diffusion counterpart in \cite{CDLL,DLZ}, \cite{DLou,KMY} for the bistable or combustion nonlinear term, \cite{DFS,CTTW} for the model with delay, \cite{ZZLD,NV,DLNZ,WND} for the epidemic models with nonlocal diffusions, \cite{DWZ,DWu,DNS} for L-V competition or predator-prey models, and \cite{PLL,ZWang} for the model with nonlocal diffusion and seasonal succession.

Before ending this introduction, we would like to discuss how we guess the correct rate of free boundary $h(t)$. As in the proof of Lemma \ref{l2.1}, we need to first estimate the orders of $\int_{1}^{h}J_i(y)y\dy$ and $h\int_{h}^{\yy}J_i(y)\dy$, and find the largest order of them. For example, if \eqref{1.6} holds, simple calculations show that the order of $h\int_{h}^{\yy}J_i(y)\dy$ is larger than that of $\int_{1}^{h}J_i(y)y\dy$, i.e.,
\[\int_{1}^{h}J_i(y)y\dy=h\int_{h}^{\yy}J_i(y)\dy o_{h\to\yy}(1).\]
Moreover,
\[h\int_{h}^{\yy}J_i(y)\dy\sim\frac{h\ln^{\alpha}\ln h}{\ln^{\beta-1}h},\]
which, combined with the free boundary condition of \eqref{1.1}, implies that
\[h'(t)\le C\frac{h(t)\ln^{\alpha}\ln h(t)}{\ln^{\beta-1}h(t)}, {\rm ~ i.e., ~ }\ln' h(t)\le C\frac{\ln^{\alpha}\ln h(t)}{\ln^{\beta-1}h(t)} ~ ~ {\rm for ~ some ~ }C>0.\]
Then we need to find a function $\bar h(t)$ such that for some $T>0$, $\ln \bar h(T)\ge \ln h(T)$ and
\[\ln' \bar h(t)\ge C\frac{\ln^{\alpha}\ln \bar h(t)}{\ln^{\beta-1}\bar h(t)}, ~ ~ \forall t>T.\]
This will be the correct rate of $h$ if we can construct an adequate lower solution $\ud h(t)$ satisfying $\ln \ud h(t)\sim \ln \bar h(t)$.

\begin{remark}\label{r1.1}

Very recently, Du and Ni (\cite{DN4}) considerably sharpened the results in \cite{DN2}. More precisely, by assuming that kernel function $J$ satisfy
\[\dd\lim_{|x|\to\yy}J(|x|)|x|^{-\gamma}=\lambda ~ ~ {\rm for ~ some ~ }\lambda>0 ~ {\rm and ~ all~ }\gamma\in(1,2],\]
they obtained the exact rate of accelerated spreading, i.e.,
\bess
\begin{cases}
\dd\lim_{t\to\yy}\frac{h(t)}{t^{\frac{1}{\gamma-1}}}=\frac{2^{2-\gamma}}{2-\gamma}\mu\lambda, ~ ~ {\rm if ~ }\gamma\in(1,2),\\
\dd\lim_{t\to\yy}\frac{h(t)}{t\ln t}=\mu\lambda, ~ ~ {\rm if ~ }\gamma=2,
\end{cases}
\eess
where $\mu$ is the expanding rate of free boundary. Inspired by their work, if we strengthen conditions \eqref{1.6}-\eqref{1.8} as above, similar results are expected to hold for \eqref{1.1}.

Moreover, if one enhances the assumptions on kernels in \cite{JGar, MaJe, XLL} or initial data in \cite{HR,Hen,FYan,SV,WX}, for example, letting $J(|x|)e^{-|x|/\ln|x|}\to\lambda>0$ as $|x|\to\yy$ or $u(0,x)x^{\alpha}\to\lambda>0$ for some $\alpha>0$ as $x\to\yy$, can we get better estimates on the location of level set? These will be very interesting.
\end{remark}

\section{Proof of Theorem \ref{t1.1}}

In this section, we will prove Theorem \ref{t1.1} by several lemmas. The arguments depend crucially on the suitably upper and lower solutions, as well as the proper comparison techniques, which is inspired by \cite{DN2,DN3,DN4,DNW}. Firstly we give the upper bounds for free boundary $h(t)$.

\begin{lemma}
\label{l2.1}Let $(u,v,h)$ be the unique solution of \eqref{1.1} and spreading happen. Then there exists a positive constant $C$ depending only on $(a,b,H,G,J_1,J_2, \mu_1,\mu_2)$ such that for all large $t$,
\bess
\begin{cases}
\ln h(t)\le C t^{\frac{1}{\beta}}\ln^{\frac{\alpha}{\beta}}t &{\rm if ~ }\eqref{1.6} ~ {\rm holds},\\
h(t)\le C t^{\frac{1}{\gamma-1}}\ln^{\frac{\beta}{\gamma-1}}t &{\rm if ~ }\eqref{1.7} ~ {\rm holds},\\
h(t)\le C t\ln^{\beta+1} t &{\rm if ~ }\eqref{1.8} ~ {\rm holds ~ with ~ }\beta>-1,\\
h(t)\le C t\ln\ln t &{\rm if ~ }\eqref{1.8} ~ {\rm holds ~ with ~ }\beta=-1.
\end{cases}
\eess
\end{lemma}
\begin{proof}
Simple computations show that for $i=1,2$,
 \bess
 \int_0^{h}\!\int_{h}^{\infty}\!J_i(x-y){\rm d}y{\rm d}x&=&\int_0^{1}J_i(y)y{\rm d}y+\int_{1}^{h}J_i(y)y{\rm d}y+h\int_{h}^{\infty}\!\!J_i(y){\rm d}y:=I(h).
 \eess
 Then by estimating the two terms $\int_{1}^{h}J_i(y)y{\rm d}y$ and $h\int_{h}^{\infty}\!\!J_i(y){\rm d}y$, we have
 \bes\label{2.1}
\begin{cases}
\dd I(h)\sim h\frac{\ln^{\alpha}\ln h}{\ln^{\beta-1} h} &{\rm if ~ }\eqref{1.6} ~ {\rm holds},\\
\dd I(h)\sim h^{2-\gamma}\ln^{\beta} h &{\rm if ~ }\eqref{1.7} ~ {\rm holds},\\
\dd I(h)\sim \ln^{\beta+1}h &{\rm if ~ }\eqref{1.8} ~ {\rm holds ~ with ~ }\beta>-1,\\
\dd I(h)\sim \ln\ln h &{\rm if ~ }\eqref{1.8} ~ {\rm holds ~ with ~ }\beta=-1.
\end{cases}
\ees
On the other hand, since there is a large $T>0$ such that $(u(t,x),v(t,x))\le(2u^*,2v^*)$ for $t\ge T$ and $x\in[0,h(t)]$, we have that for $t\ge T$,
\[h'(t)\le\int_0^{h(t)}\!\!\!\int_{h(t)}^{\infty}\!2\big(\mu_1u^*J_1(x-y)+\mu_2v^*J_2(x-y)\big){\rm d}y{\rm d}x,\]
which, combined with \eqref{2.1}, implies that there exists a positive constant $C_1$ relying only on $(a,b,H,G,J_1,J_2, \mu_1,\mu_2)$ such that for $t\ge T$,
 \bess
\begin{cases}
\dd h'(t)\le C_1 h(t)\frac{\ln^{\alpha}\ln h(t)}{\ln^{\beta-1} h(t)} &{\rm if ~ }\eqref{1.6} ~ {\rm holds},\\
\dd h'(t)\le C_1h^{2-\gamma}(t)\ln^{\beta} h(t) &{\rm if ~ }\eqref{1.7} ~ {\rm holds},\\
\dd h'(t)\le C_1 \ln^{\beta+1}h(t) &{\rm if ~ }\eqref{1.8} ~ {\rm holds ~ with ~ }\beta>-1,\\
\dd h'(t)\le C_1\ln\ln h(t) &{\rm if ~ }\eqref{1.8} ~ {\rm holds ~ with ~ }\beta=-1.
\end{cases}
\eess

We next only handle the case where \eqref{1.7} holds since other cases can be done similarly.  Let
\[\bar{h}(t)=K(t+\theta)^{\frac{1}{\gamma-1}}\ln^{\frac{\beta}{\gamma-1}}(t+\theta), ~ ~ t\ge T,\]
where $\theta$ is a positive constant to be determined later and $K$ is large enough such that $K^{\gamma-1}(\gamma-1)^{\beta-1}\ge 2C_1$.
Then it is not hard to show that by choosing $\theta$ to be large enough, we have
 \bess\bar{h}'(t)\ge K\frac{1}{\gamma-1}(t+\theta)^{\frac{2-\gamma}{\gamma-1}}\ln^{\frac{\beta}{\gamma-1}}(t+\theta)\ge C_1 \bar h^{2-\gamma}(t)\ln^{\beta} \bar h(t) ~ ~ {\rm for ~ }t\ge T.
 \eess
Moreover, we may let $\theta$ be sufficiently large if necessary such that $\bar{h}(T)>h(T)$. A simple comparison argument yields $h(t)\le \bar h(t)$ for $t\ge T$. Thus the proof is complete.
\end{proof}

 The remainder of this section is devoted to the proof of lower bounds. Firstly, inspired by \cite[Lemma 2.2]{DN4}, we give a general technical lemma which is crucial for the construction of the desired lower solutions.

 \begin{proposition}\label{p2.1}Assume that positive constants $k_1,k_2$ and $l$ satisfy $k_2>k_1>l\gg1$. Let $J$ satisfy {\bf (J)} and $\rho\ge1$ be a constant. Define
 \[\xi(x)=\min\kk\{1,\;\bigg(\frac{k_2-x}{k_1}\bigg)^{\rho}\rr\}.\]
Then for any given small $\ep>0$, there exists a $k_0>l$, depending only on $l, \rho, \ep$ and $J$, such that when $k_1>k_0$ and $k_2-k_1>2k_0$, we have
 \[\int_{l}^{k_2}J(x-y)\xi(y)\dy\ge (1-\ep)\xi(x) ~ ~ {\rm for ~ }x\in[k_0,k_2].\]
 \end{proposition}

 \begin{proof} Due to {\bf (J)}, for any given small $\ep>0$, there exists a $l_0\gg1$ such that
 \[\dd\int_{-l_0}^{l_0}J(y)\dy\ge 1-\frac{\ep}{2}.\]
 If we extend $\xi(x)$ by $\xi(x)\equiv0$ for $x\ge k_2$, it is easy to verify that $\xi(x)$ is convex in $x\ge k_2-k_1$ and
 \bes\label{2.2}
 |\xi(x)-\xi(y)|\le \frac{\rho}{k_1}|x-y| ~ ~ {\rm for ~ } x, y\in[0,\yy).
 \ees

 Set $k_0>2l_0$, $k_1>k_0$ and $k_2-k_1>2k_0$. We now complete the proof by
  \bess
&& {\bf Case ~  1:} ~ x\in[k_0,k_2-k_1-l_0]; ~ ~ ~ ~ ~ ~  ~ ~ {\bf Case ~  2:} ~ x\in[k_2-k_1-l_0,k_2-k_1+l_0];\\
 && {\bf Case ~  3:} ~ x\in[k_2-k_1+l_0,k_2-l_0]; ~ ~ {\bf Case  ~ 4:} ~ x\in[k_2-l_0,k_2].\eess

For Case 1, straightforward computations yield
 \bess
 \dd\int_{l}^{k_2}J(x-y)\xi(y)\dy&\ge&\int_{l}^{k_2-k_1}J(x-y)\dy
 =\int_{l-x}^{k_2-k_1-x}J(y)\dy\\
 &\ge&\int_{l-k_0}^{l_0}J(y)\dy\ge\int_{-l_0}^{l_0}J(y)\dy\\
 &\ge&1-\frac{\ep}{2}\ge(1-\ep)\xi(x).
 \eess

For Case 2, in light of \eqref{2.2}, we have
 \bess
 \int_{l}^{k_2}J(x-y)\xi(y)\dy&=&\int_{l-x}^{k_2-x}J(y)\xi(x+y)\dy
 \ge\int_{l-k_2+k_1+l_0}^{k_1-l_0}J(y)\xi(x+y)\dy\\
 &\ge&\int_{-l_0}^{l_0}J(y)\xi(x+y)\dy\\
 &=&\int_{-l_0}^{l_0}J(y)\xi(x)\dy+\int_{-l_0}^{l_0}J(y)\left[\xi(x+y)-\xi(x)\right]\dy\\
 &\ge&\int_{-l_0}^{l_0}J(y)\xi(x)\dy-\frac{\rho l_0}{k_1}\ge(1-\frac{\ep}{2})\xi(x)-\frac{\rho l_0}{k_1}\\
 &=&(1-\ep)\xi(x)+\frac{\ep}{2}\xi(x)-\frac{\rho l_0}{k_1}\\
 &\ge&(1-\ep)\xi(x)+\frac{\ep}{2^{\rho+1}}-\frac{\rho l_0}{k_1}\ge(1-\ep)\xi(x)
 \eess
 provided that $k_1\ge\frac{2^{\rho+1}\rho l_0}{\ep}$.

For Case 3, clearly, we have that for any $y\in[0,l_0]$, $x+y\ge x-y\ge k_2-k_1$.
Note that $\xi(x)$ is a convex function of $x$ for $x\ge k_2-k_1$. Then we can deduce
 \bess
 \dd\int_{l}^{k_2}J(x-y)\xi(y)\dy&=&\int_{l-x}^{k_2-x}J(y)\xi(x+y)\dy
 \ge\int_{l-k_2+k_1-l_0}^{l_0}J(y)\xi(x+y)\dy\\
 &\ge&\int_{-l_0}^{l_0}J(y)\xi(x+y)\dy
 =\int_{0}^{l_0}J(y)[\xi(x+y)+\xi(x-y)]\dy\\
 &\ge&2\int_{0}^{l_0}J(y)\xi(x)\dy\ge(1-\ep)\xi(x).
 \eess

For Case 4, since $\xi(x)\equiv0$ for $x\ge k_2$, we see
 \bess
 \dd\int_{l}^{k_2}J(x-y)\xi(y)\dy&=&\int_{l-x}^{k_2-x}J(y)\xi(x+y)\dy
 \ge\int_{-l_0}^{k_2-x}J(y)\xi(x+y)\dy\\
 &=&\int_{-l_0}^{l_0}J(y)\xi(x+y)\dy-\int_{k_2-x}^{l_0}J(y)\xi(x+y)\dy\\
 &=&\int_{-l_0}^{l_0}J(y)\xi(x+y)\dy.
 \eess
Then by the similar arguments as in Case 3, we can prove Case 4. Thus setting $k_0=\frac{2^{\rho+1}\rho l_0}{\ep}$, we finish the proof.
 \end{proof}

The following comparison principle for free boundary problem with nonlocal diffusions will be used often in the later discussion.
\begin{proposition}{\rm ({\rm \cite[Proposition 2.3]{LL2} })}\label{p2.2} For any $T>0$, we assume that $P_i$ with $i=1,2$ satisfy {\bf (J)}, $s(t)$ and $r(t)$ are continuous and increasing in $[0,T]$, as well as $\underline{h}(t),\; \bar{h}(t)\in C^1([0,T])$ with $s(t)\le r(t)\le \min\{\underline{h}(t),\bar{h}(t)\}$ for $t\in[0,T]$. Let $(\underline{u},\ud v)\in [C([0,T]\times[s(t),\underline{h}(t)])]^2$, $(\bar{u},\bar{v})\in [C([0,T]\times[s(t),\bar{h}(t)])]^2$, the left derivatives $(\underline{u}_t(t-0,x), \ud v_t(t-0,x))$ and $(\bar{u}_t(t-0,x), \bar{v}_t(t-0,x))$ exist in $(0,T]\times[r(t),\underline{h}(t)]$ and $(0,T]\times[r(t),\bar{h}(t)]$, respectively. Suppose $(\underline{u}(t,x),\ud v(t,x))\le (\bar u(t,x),\bar{v}(t,x))$ for $0<t\le T$ and $ x\in[s(t),r(t)]$.
Moreover, $(\underline{u},\ud v, \underline{h})$ and $(\bar{u},\bar{v}, \bar{h})$ satisfy
\bess
\left\{\begin{aligned}
&\underline u_t(t-0,x)\le d_1\displaystyle\int_{s(t)}^{\underline h(t)}P_1(x-y)\underline u(t,y){\rm d}y-d_1\underline u-a\ud u+H(\ud v), && 0<t\le T,~x\in[r(t),\underline h(t)),\\
&\underline v_t(t-0,x)\le d_2\displaystyle\int_{s(t)}^{\underline h(t)}P_2(x-y)\underline v(t,y){\rm d}y-d_2\underline v-b\ud v+G(\ud u), && 0<t\le T,~x\in[r(t),\underline h(t)),\\
&\underline u(t,\underline h(t))\le0,~ \underline v(t,\underline h(t))\le0, && 0<t\le T,\\
&\underline h'(t)<\displaystyle\int_{s(t)}^{\underline h(t)}\int_{\underline h(t)}^{\infty}
\big(\mu_1P_1(x-y)\underline u(t,x)+\mu_2P_2(x-y)\ud v(t,x)\big){\rm d}y{\rm d}x,&& 0<t\le T,\\
&\underline h(0)<\bar h(0),\;\;\underline u(0,x)\le \bar u(0,x), ~ \ud v(0,x)\le \bar{v}(0,x), && x\in[s(0),\underline h(0)]
\end{aligned}\right.
\eess
and
\bess
\left\{\begin{aligned}
&\bar u_t(t-0,x)\ge d_1\displaystyle\int_{s(t)}^{\bar h(t)}P_1(x-y)\bar u(t,y){\rm d}y-d_1\bar u-a\bar u+H(\bar v), && 0<t\le T,~x\in[r(t),\bar h(t)),\\
&\bar v_t(t-0,x)\ge d_2\displaystyle\int_{s(t)}^{\bar h(t)}P_2(x-y)\bar v(t,y){\rm d}y-d_2\bar v-b\bar v+G(\bar u), && 0<t\le T,~x\in[r(t),\bar h(t)),\\
&\bar u(t,\bar h(t))\ge0,~ \bar v(t,\bar h(t))\ge0, && 0<t\le T,\\
&\bar h'(t)\ge\displaystyle\int_{s(t)}^{\bar h(t)}\int_{\bar h(t)}^{\infty}
\big(\mu_1P_1(x-y)\bar u(t,x)+\mu_2P_2(x-y)\bar v(t,x)\big){\rm d}y{\rm d}x,&& 0<t\le T,\\
\end{aligned}\right.
\eess
respectively. Then we have
\[\underline{h}(t)\le\bar{h}(t), ~ ~ ~( \underline{u}(t,x),\ud v(t,x))\le (\bar{u}(t,x),\bar{v}(t,x))\;\; ~ {\rm for ~ }\;(t,x)\in[0,T]\times[s(t),\underline{h}(t)].\]
\end{proposition}
 Using $\mathcal{R}_0>1$ and {\bf (H)}, we can find a constant $r>0$ depending only on $(a,b,H,G)$ such that
 \bes\label{2.3}
 (-a\eta u^*+H(\eta v^*), -b\eta v^*+G(\eta u^*))\ge \min\{\eta,1-\eta\}(r,r) ~ ~ {\rm for ~ all ~ }\eta\in[0,1].
 \ees

\begin{lemma}\label{l2.2}Let $(u,v,h)$ be the unique solution of \eqref{1.1} and the assumption \eqref{1.6} hold. Then the following statements are valid.
 \begin{enumerate}[$(1)$]
 \item There exists a positive constant $C$ depending only on the initial functions, parameters and kernel functions of \eqref{1.1} such that $\ln h(t)\ge C t^{\frac{1}{\beta}}\ln^{\frac{\alpha}{\beta}}t$ for all large $t$.
 \item For any $s(t)$ satisfying $\ln s(t)=o(t^{\frac{1}{\beta}}\ln^{\frac{\alpha}{\beta}}t)$, we have
 \[\liminf_{t\to\yy}(u(t,x),v(t,x))\ge (U(x),V(x)) ~ ~ {\rm uniformly ~ in ~ }[0,s(t)],\]
 where $(U(x),V(x))$ is the unique bounded positive solution of \eqref{1.4}.
 \end{enumerate}
 \end{lemma}
 \begin{proof}For any given $0<\ep\ll1$, according to \cite[Proposition 2.2]{LL1}, we can find a $X_{\ep}>0$ such that
\[(U(x),V(x))\ge (1-\frac{\sqrt{\ep}}{2})(u^*,v^*)~ ~ {\rm  for ~ }x\ge X_{\ep}.\]

 For positive constants $l_1\ll1$ and $\theta\gg1$ to be determined later, we define
 \[\underline{h}(t)=e^{l_1(t+\theta)^{\frac{1}{\beta}}\ln^{\frac{\alpha}{\beta}}(t+\theta)}, ~ \underline{u}(t,x)=u^*_{\ep}\zeta(t,x), ~ \underline{v}(t,x)=v^*_{\ep}\zeta(t,x) ~ ~ {\rm for ~ }(t,x)\in[0,\yy)\times[X_{\ep},\underline{h}(t)],\]
 where  $e^{l_1\theta^{\frac{1}{\beta}}\ln^{\frac{\alpha}{\beta}}\theta}\ge4$, $(u^*_{\ep},v^*_{\ep})=(1-\sqrt{\ep})(u^*,v^*)$, and
 \[\zeta(t,x):=\min\kk\{1,\;\left[8\frac{\underline{h}(t)-x}{\underline{h}(t)}\right]^{\rho}\rr\}\]
 with $\rho\gg1$ satisfying
 \[1-\frac{1}{\ln^{\frac{1}{\rho-1}}4}\ge\frac{7}{8}, ~ ~ 1-\beta+\frac{1}{\rho-1}\le0.\]
Assume $\ln \theta>|\alpha|$. Thus $\ud h(t)$ is strictly increasing for $t\ge0$.  Making use of Proposition \ref{p2.1} with $l=X_{\ep}$, we have that there exists a large $k_0>0$ such that if $e^{l_1\theta^{\frac{1}{\beta}}\ln^{\frac{\alpha}{\beta}}\theta}>8k_0$, then
 \bes\label{2.4}
 \int_{X_{\ep}}^{\underline{h}(t)}J_i(x-y)\zeta(t,y)\dy\ge(1-\ep^2)\zeta(t,x) ~ ~ {\rm for ~ }t>0, ~ x\in[k_{0},\underline{h}(t)], ~ i=1,2.
 \ees

Next we will prove that by choosing $l_1$, $\theta$ and $T$ suitably, there holds:
  \bes\label{2.5}
\left\{\begin{aligned}
&\underline u_t\le d_1\displaystyle\int_{X_{\ep}}^{\underline h(t)}J_1(x-y)\underline u(t,y){\rm d}y-d_1\underline u-a\underline u+H(\ud v),\\
&\hspace{50mm}t>0,~x\in[k_0,\underline h(t))\setminus\left\{\frac{7\underline{h}(t)}8\right\},\\
&\underline v_t\le d_2\displaystyle\int_{X_{\ep}}^{\underline h(t)}J_2(x-y)\underline v(t,y){\rm d}y-d_2\underline v-b\underline v+G(\ud u),\\
&\hspace{50mm} t>0,~x\in[k_0,\underline h(t))\setminus\left\{\frac{7\underline{h}(t)}8\right\},\\
&\underline u(t,\underline h(t))=\ud v(t,\ud h(t))=0,&& t>0,\\
&\underline h'(t)<\displaystyle\int_{X_{\ep}}^{\underline h(t)}\int_{\underline h(t)}^{\infty}
\big(\mu_1J_1(x-y)\underline u(t,x)+\mu_2J_2(x-y)\ud v(t,x)\big){\rm d}y{\rm d}x,&& t>0,\\
&\underline{u}(t,x)\le u(t+T,x), ~ \ud v(t,x)\le v(t+T,x), && t>0, ~ x\in[X_{\ep},k_0],\\
&\underline h(0)<h(T),\;\;\underline u(0,x)\le u(T,x),~ \ud v(0,x)\le v(T,x), && x\in[X_{\ep},\underline h(0)].
\end{aligned}\right.
\ees
Once we have proved \eqref{2.5}, by virtue of a comparison argument (Proposition \ref{p2.2} with $(s(t),r(t))=(X_{\ep},k_0)$), we derive
\[h(t+T)\ge \underline{h}(t) ~ ~ {\rm and ~ ~ }(u(t+T,x),v(t+T))\ge (\underline u(t,x),\ud v(t,x)) ~ ~ {\rm for ~ }t\ge0, ~ x\in[X_{\ep},\underline{h}(t)],\]
which obviously yields the assertion (1). As for the assertion (2), we have
\bess
&&\max_{x\in[X_{\ep},s(t)]}|\underline{u}(t,x)-U(x)|\\
 &&\le u^*(1-\sqrt{\ep})\kk(1-\min\kk\{1,\;
\left[8\frac{\underline{h}(t)-s(t)}{\underline{h}(t)}\right]^{\rho}\rr\}\rr)
+\frac{3\sqrt{\ep}u^*}{2}\\
&&=u^*(1-\sqrt{\ep})\left[1-\min\kk\{1,\;
8^{\rho}\left[1-\frac{e^{o_{t\to\yy}(1)t^{\frac{1}{\beta}}\ln^{\frac{\alpha}{\beta}}t} }{\underline{h}(t)}\right]^{\rho}\rr\}\right]
+\frac{3\sqrt{\ep}u^*}{2}\to \frac{3\sqrt{\ep}u^*}{2} {\rm ~ as ~ }t\to\yy.
\eess
Thus there exists a $T_1>T$ such that
\[\underline{u}(t,x)\ge U(x)-2\sqrt{\ep}u^* ~ ~ {\rm for ~ }t\ge T_1, ~ x\in[X_{\ep},s(t)].\]
Recall that $u(t+T,x)\ge \underline u(t,x)$ for $t\ge0$ and $x\in[X_{\ep},\underline{h}(t)]$. We can conclude that
\[u(t,x)\ge U(x)-2\sqrt{\ep}u^* ~ ~ {\rm  for ~ }t\ge T_1+T, ~  x\in[X_{\ep},s(t)].\]
 Moreover, since $u(t,x)\to U(x)$ locally uniformly in $[0,\yy)$ as $t\to\yy$, one can choose a $T_2>T_1+T$ such that
 \[u(t,x)\ge U(x)-2\sqrt{\ep}u^*~ ~ {\rm for ~ }t\ge T_2, ~ x\in[0,X_{\ep}].\]
  Thus for any small $\ep>0$, we have
   \[u(t,x)\ge U(x)-2\sqrt{\ep}u^* ~ ~ {\rm for ~ }t\ge T_2, ~ x\in[0,s(t)],\]
    which immediately implies $\liminf_{t\to\yy}u(t,x)\ge U(x)$ uniformly in $[0,s(t)]$. Analogously, we can show $\liminf_{t\to\yy}v(t,x)\ge V(x)$ uniformly in $[0,s(t)]$. Thus the assertion (2) is obtained.

Now it remains to verify \eqref{2.5}. Firstly, for $x\in[7\ud h(t)/8,\ud h(t)]$, if $e^{l_1\theta^{\frac{1}{\beta}}\ln^{\frac{\alpha}{\beta}}\theta}\gg1$, we have
\bes\label{2.6}
\int_{X_{\ep}}^{\ud h(t)}J_i(x-y)\zeta(t,y)\dy&\ge&\int_{X_{\ep}}^{\frac{7\ud h(t)}{8}}J_i(x-y)\zeta(t,y)\dy=\int_{X_{\ep}}^{\frac{7\ud h(t)}{8}}J_i(x-y)\dy\nonumber\\
&=&\int_{X_{\ep}-x}^{\frac{7\ud h(t)}{8}-x}J_i(y)\dy\nonumber\\
&\ge&\int_{\frac{\ud h(t)}{8}}^{\frac{\ud h(t)}{4}}J_i(y)\dy\nonumber\\
&\ge&\varsigma_1\int_{\frac{\ud h(t)}{8}}^{\frac{\ud h(t)}{4}}\frac{\ln^{\alpha}\ln y}{y\ln^{\beta}y}{\rm d}y\dy\nonumber\\
&=&\frac{\varsigma_1}{\beta-1}\left[\frac{\ln^{\alpha}\ln\frac{\ud h(t)}{8}}{\ln^{\beta-1}\frac{\ud h(t)}{8}}-\frac{\ln^{\alpha}\ln\frac{\ud h(t)}{4}}{\ln^{\beta-1}\frac{\ud h(t)}{4}}\right]+\frac{\varsigma_1\alpha}{\beta-1}\int_{\ln\frac{\ud h(t)}{8}}^{\ln\frac{\ud h(t)}{4}}\frac{\ln^{\alpha-1}y}{y^{\beta}}\dy\nonumber\\
&\ge&\frac{\varsigma_1\ln2\ln^{\alpha}\ln\ud h(t)}{2\ln^{\beta}\ud h(t)}+\frac{\varsigma_1\alpha}{\beta-1}\int_{\ln\frac{\ud h(t)}{8}}^{\ln\frac{\ud h(t)}{4}}\frac{\ln^{\alpha-1}y}{y^{\beta}}\dy\nonumber\\
&=&\frac{\ln^{\alpha}\ln\ud h(t)}{\ln^{\beta}\ud h(t)}\left(\frac{\varsigma_1\ln2}{2}+o_{\ud h\to\yy}(1)\right)\nonumber\\
&\ge&\frac{\varsigma_1\ln2\ln^{\alpha}\ln\ud h(t)}{4\ln^{\beta}\ud h(t)}\nonumber\\
&\ge&\frac{\varsigma_1\ln2\min\{1,(\frac{1}{\beta})^{\alpha}\}\ln^{\alpha}(t+\theta)}{4\ln^{\beta}\ud h(t)},
\ees
where we have used that
\bess
 \frac{\ln^{\alpha}\ln\frac{\ud h(t)}{8}}{\ln^{\beta-1}\frac{\ud h(t)}{8}}-\frac{\ln^{\alpha}\ln\frac{\ud h(t)}{4}}{\ln^{\beta-1}\frac{\ud h(t)}{4}}=\frac{\ln^{\alpha}\ln\ud h(t)}{\ln^{\beta}\ud h(t)}\left[(\beta-1)\ln2 +o(1)\right],
 \eess
 and $\ln^{\alpha}\ln\ud h(t)\ge\min\{1,(\frac{1}{\beta})^{\alpha}\}\ln^{\alpha}(t+\theta)$.
 From \eqref{2.3}, it follows that
\bes\label{2.7}
(-a\ud u+H(\ud v),-b\ud v+G(\ud u))\ge\tilde{r}\ep(\ud u,\ud v) \;~ ~ {\rm for ~  }x\in[k_0,\ud h(t)], ~ {\rm if } ~ \ep ~ {\rm is} ~ {\rm small} ~ {\rm enough},
 \ees
 where $\tilde{r}$ depends only on $a$, $b$, $H$ and $G$.

Hence, for $x\in[k_0,\underline{h}(t)]$, using \eqref{2.4} and \eqref{2.7}, if $\ep$ is suitably small, we have
\bes\label{2.8}
 &&d_1\displaystyle\int_{X_{\ep}}^{\underline h(t)}J_1(x-y)\underline u(t,y){\rm d}y-d_1\underline u-a\ud u+H(\ud v)\nonumber\\
&\ge&\left[\frac{\tilde{r}\ep}{2}+\left(d_1-\frac{\tilde{r}\ep}2\right)\right]
\int_{X_{\ep}}^{\underline h(t)}J_1(x-y)\underline u(t,y){\rm d}y-\left(d_1-\tilde{r}\ep\right)\underline u\nonumber\\
&\ge&\frac{\tilde{r}\ep}{2}\int_{X_{\ep}}^{\underline h(t)}J_1(x-y)\underline u(t,y){\rm d}y+\left(d_1-\frac{\tilde{r}\ep}2\right)(1-\ep^2)\underline{u}
-\left(d_1-\tilde{r}\ep\right)\underline u\nonumber\\
&\ge&\frac{\tilde{r}\ep}{2}\int_{X_{\ep}}^{\underline h(t)}J_1(x-y)\underline u(t,y){\rm d}y+\frac{\tilde{r}\ep}{3}\ud u +\left(d_1-\frac{\tilde{r}\ep}2\right)(1-\ep^2)\underline{u}
-\left(d_1-\frac{2\tilde{r}\ep}{3}\right)\underline u\nonumber\\
&\ge&\frac{\tilde{r}\ep}{2}\int_{X_{\ep}}^{\underline h(t)}J_1(x-y)\underline u(t,y){\rm d}y+\frac{\tilde{r}\ep}{3}\ud u.
\ees
On the other hand, we have $\underline{u}_t(t,x)=0$ for $t>0$ and $x\in[k_0,{7\underline{h}(t)}/8]$. Due to \eqref{2.8}, the first inequality of \eqref{2.5} holds when $t>0$ and $x\in[k_0,{7\underline{h}(t)}/8)$. For $t>0$ and
\[x\in\left[\left(1-\frac{1}{\ln^{\frac{1}{\rho-1}}\ud h(t)}\right)\underline{h}(t),\ud h(t)\right],\]
in view of \eqref{2.6} and \eqref{2.8}, we have
\bess
\underline{u}_t(t,x)&=&u^*_{\ep}8^{\rho}\rho\left(\frac{\ud h(t)-x}{\ud h(t)}\right)^{\rho-1}\frac{x\ud h'(t)}{\ud h^2(t)}\\
&\le&u^*_{\ep}8^{\rho}\rho\left(\frac{\ud h(t)-x}{\ud h(t)}\right)^{\rho-1}\frac{\ud h'(t)}{\ud h(t)}\\
&=&\frac{u^*_{\ep}8^{\rho}\rho l_1}{\beta}\left(\frac{\ud h(t)-x}{\ud h(t)}\right)^{\rho-1}(t+\theta)^{\frac{1-\beta}{\beta}}\ln^{\frac{\alpha}{\beta}}(t+\theta)\left[1+\alpha\ln^{-1}(t+\theta)\right]\\
&\le&\frac{2u^*_{\ep}8^{\rho}\rho l_1}{\beta}\left(\frac{\ud h(t)-x}{\ud h(t)}\right)^{\rho-1}(t+\theta)^{\frac{1-\beta}{\beta}}\ln^{\frac{\alpha}{\beta}}(t+\theta)\\
&=&\frac{2u^*_{\ep}8^{\rho}\rho l^{\beta}_1}{\beta}\left(\frac{\ud h(t)-x}{\ud h(t)}\right)^{\rho-1}\ln^{1-\beta}\ud h(t)\ln^{\alpha}(t+\theta)\\
&\le&\frac{2u^*_{\ep}8^{\rho}\rho l^{\beta}_1}{\beta}\frac{\ln^{\alpha}(t+\theta)}{\ln^{\beta}\ud h(t)}\\
&\le&\frac{u^*_{\ep}\tilde{r}\ep\varsigma_1\ln2\min\{1,(\frac{1}{\beta})^{\alpha}\}\ln^{\alpha}(t+\theta)}{8\ln^{\beta}\ud h(t)}
\eess
provided that $l_1$ is suitably small such that
\[\frac{2u^*_{\ep}8^{\rho}\rho l^{\beta}_1}{\beta}\le\frac{u^*_{\ep}\tilde{r}\ep\varsigma_1\ln2\min\{1,(\frac{1}{\beta})^{\alpha}\}}{8}.\]
 For $t>0$ and
\[x\in[\frac{7\ud h(t)}{8},\left(1-\frac{1}{\ln^{\frac{1}{\rho-1}}\ud h(t)}\right)\underline{h}(t)],\]
using \eqref{2.8} yields
\bess
\underline{u}_t(t,x)&\le&\rho(\frac{\ud h(t)-x}{\ud h(t)})^{-1}\frac{\ud h'(t)}{\ud h(t)}\ud u\\
&\le&\frac{2\rho l^{\beta}_1}{\beta}\ln^{1-\beta+\frac{1}{\rho-1}}\ud h(t)\ln^{\alpha}(t+\theta)\ud u\\
&\le&\frac{2\rho l^{1+\frac{1}{\rho-1}}_1}{\beta}\ud u\le\frac{\tilde{r}\ep}{3}\ud u
\eess
if $l_1$ is small enough such that
\[\frac{2\rho l^{1+\frac{1}{\rho-1}}_1}{\beta}\le\frac{\tilde{r}\ep}{3}.\]
 Hence the first inequality of \eqref{2.5} holds with $l_1$, $\theta$ and $\ep$ chosen as above. Analogously, we can show the inequality in the third line of \eqref{2.5} is valid if $l_1$, $\theta$ and $\ep$ are chosen suitably.

Then we prove the inequality in the six line of \eqref{2.5} holds. Straightforward computations show that if $e^{l_1\theta^{\frac{1}{\beta}}\ln^{\frac{\alpha}{\beta}}\theta}$ is suitably large, then
\bess
&&\displaystyle\int_{X_{\ep}}^{\underline h(t)}\int_{\underline h(t)}^{\infty}
\big(\mu_1J_1(x-y)\underline u(t,x)+\mu_2J_2(x-y)\ud v(t,x)\big){\rm d}y{\rm d}x\\
&&\ge\int_{X_{\ep}}^{\frac{7\underline h(t)}{8}}\int_{\underline h(t)}^{\infty}
\big(\mu_1u^*_{\ep}J_1(x-y)+\mu_2v^*_{\ep}J_2(x-y)\big){\rm d}y{\rm d}x\\
&&=\int_{X_{\ep}-\ud h(t)}^{\frac{-\ud h(t)}{8}}\int_{0}^{\infty}
\big(\mu_1u^*_{\ep}J_1(x-y)+\mu_2v^*_{\ep}J_2(x-y)\big){\rm d}y{\rm d}x\\
&&=\left\{\int_{\frac{\ud h(t)}{8}}^{\ud h(t)-X_{\ep}}\int_{\frac{\ud h(t)}{8}}^{y}+\int_{\ud h(t)-X_{\ep}}^{\yy}\int_{\frac{\ud h(t)}{8}}^{\ud h(t)-X_{\ep}}\right\}
\big(\mu_1u^*_{\ep}J_1(y)+\mu_2v^*_{\ep}J_2(y)\big){\rm d}x{\rm d}y\\
&&\ge\int_{\ud h(t)-X_{\ep}}^{\yy}\int_{\frac{\ud h(t)}{8}}^{\ud h(t)-X_{\ep}}
\big(\mu_1u^*_{\ep}J_1(y)+\mu_2v^*_{\ep}J_2(y)\big){\rm d}x{\rm d}y\\
&&\ge\frac{\ud h(t)}{2}\int_{\ud h(t)-X_{\ep}}^{\yy}
\big(\mu_1u^*_{\ep}J_1(y)+\mu_2v^*_{\ep}J_2(y)\big){\rm d}y\\
&&\ge C_1\ud h(t)\int_{\ud h(t)}^{\yy}
\frac{\ln^{\alpha}\ln y}{y\ln^{\beta}y}{\rm d}y\\
&&=\frac{C_1\ud h(t)\ln^{\alpha}\ln \ud h(t)}{(\beta-1)\ln^{\beta-1}\ud h(t)}+\frac{\alpha C_1\ud h(t)}{\beta-1}\int_{\ln \ud h(t)}^{\yy}y^{-\beta}\ln^{\alpha-1}y\dy\\
&&=\frac{\ud h(t)\ln^{\alpha}\ln \ud h(t)}{\ln^{\beta-1}\ud h(t)}(\frac{C_1}{\beta-1}+o_{\ud h\to\yy}(1))\\
&&\ge\frac{C_1\ud h(t)\ln^{\alpha}\ln \ud h(t)}{2(\beta-1)\ln^{\beta-1}\ud h(t)}\\
&&\ge\frac{C_1\min\{1,(\frac{1}{\beta})^{\alpha}\}\ud h(t)\ln^{\alpha}(t+\theta)}{2(\beta-1)\ln^{\beta-1}\ud h(t)},
\eess
where $C_1$ depends only on $a$, $b$, $H$, $G$, $\mu_i$ and $J_i$ with $i=1,2$.
 On the other hand, simple computations arrive at
\bess
\underline{h}'(t)&=&\frac{\ud h(t)\ln \ud h(t)}{\beta(t+\theta)}\left[1+\alpha\ln^{-1}(t+\theta)\right]\\
&\le&\frac{2\ud h(t)\ln \ud h(t)}{\beta(t+\theta)}\\
&=&\frac{2l^{\beta}_1\ud h(t)\ln^{\alpha}(t+\theta)}{\beta\ln^{\beta-1}\ud h(t)}\\
&\le&\frac{C_1\min\{1,(\frac{1}{\beta})^{\alpha}\}\ud h(t)\ln^{\alpha}(t+\theta)}{2(\beta-1)\ln^{\beta-1}\ud h(t)}
\eess
provided that $l_1$ is suitably small such that
\[\frac{2l^{\beta}_1}{\beta}\le \frac{C_1\min\{1,(\frac{1}{\beta})^{\alpha}\}}{2(\beta-1)}.\]
Thus the inequality in the fifth line of \eqref{2.5} holds.

Now to prove \eqref{2.5}, it remains to show the inequalities in the last two lines of \eqref{2.5}.  For positive constants $\ep,\;\theta$ and $l_1$ chosen as above, since spreading happens for \eqref{1.1}, we can find a $T>0$ such that $h(T)>e^{l_1\theta^{\frac{1}{\beta}}\ln^{\frac{\alpha}{\beta}}\theta}=\underline{h}(0)$ and
\[(u(t,x),v(t,x))\ge (1-\frac{\sqrt{\ep}}{2})(U(x),V(x)) ~ ~ {\rm  for ~ }t\ge T, ~ x\in[X_{\ep}, \underline{h}(0)].\]
 Recalling $(U(x),V(x))\ge (1-\frac{\sqrt{\ep}}{2})(u^*,v^*)$  for $x\ge X_{\ep}$, we have that when $t\ge T$ and $x\in[X_{\ep}, \underline{h}(0)]$,
\[(u(t,x),v(t,x))\ge (1-\frac{\sqrt{\ep}}{2})(U(x),V(x))\ge (1-\sqrt{\ep})(u^*,v^*)\ge (\underline{u}(t,x),\ud v(t,x)),\]
which immediately implies that $(u(t+T,x),v(t+T,x))\ge(\underline{u}(t,x),\ud v(t,x))$ for $t\ge0$ and $x\in[X_{\ep},\underline{h}(0)]$. Thus \eqref{2.5} holds. The proof is complete.
 \end{proof}

 Next we consider the case where \eqref{1.7} holds.

 \begin{lemma}\label{l2.3}Let $(u,v,h)$ be the unique solution of \eqref{1.1} and the assumption \eqref{1.7} hold. Then the following statements are valid.
 \begin{enumerate}[$(1)$]
 \item There exists a positive constant $C$ such that $h(t)\ge C t^{\frac{1}{\gamma-1}}\ln^{\frac{\beta}{\gamma-1}}t$ for all large $t$.
 \item For any $s(t)=o(t^{\frac{1}{\gamma-1}}\ln^{\frac{\beta}{\gamma-1}}t)$, we have
 \[\liminf_{t\to\yy}(u(t,x),v(t,x))\ge (U(x),V(x)) ~ ~ {\rm uniformly ~ in ~ }[0,s(t)].\]
 \end{enumerate}
 \end{lemma}
 \begin{proof}
For any small $0<\ep\ll1$, let $X_{\ep}$ be defined as in the proof of Lemma \ref{l2.2}.
 For positive constants $l_1\ll1$ and $\theta\gg1$ to be determined later, we define
 \[\underline{h}(t)=l_1(t+\theta)^{\frac{1}{\gamma-1}}\ln^{\frac{\beta}{\gamma-1}}(t+\theta), ~ \underline{u}(t,x)=u^*_{\ep}\hat\zeta(t,x), ~ \underline{v}(t,x)=v^*_{\ep}\hat\zeta(t,x),\]
 where $(t,x)\in[0,\yy)\times[X_{\ep},\underline{h}(t)]$, $(u^*_{\ep},v^*_{\ep})$ is the same with that of Lemma \ref{l2.2}, and
 \[\hat{\zeta}(t,x)=\min\left\{1,2\frac{\ud h(t)-x}{\ud h(t)}\right\}.\]
In light of Proposition \ref{p2.1} with $l=X_{\ep}$, we have that there exists a large $k_0>0$ such that if $l_1\theta^{\frac{1}{\gamma-1}}\ln^{\frac{\beta}{\gamma-1}}\theta>4k_0$, then
 \bes\label{2.9}
 \int_{X_{\ep}}^{\underline{h}(t)}J_i(x-y)\hat\zeta(t,y)\dy\ge(1-\ep^2)\hat\zeta(t,x) ~ ~ {\rm for ~ }t>0, ~ x\in[k_{0},\underline{h}(t)], ~ i=1,2.
 \ees
Setting $\beta+\ln \theta>0$ ensures that $\ud h(t)$ is strictly increasing for $t\ge0$. Next we prove that by choosing $l_1$, $\theta$ and $T$ properly, \eqref{2.5} holds but with
\[x\in[k_0,\underline h(t))\setminus\left\{\frac{7\underline{h}(t)}8\right\}\]
replaced by
\[x\in[k_0,\underline h(t))\setminus\left\{\frac{\underline{h}(t)}2\right\}\]
in the first two inequalities. Once we have done it, arguing as in the proof of Lemma \ref{l2.2}, we can get the desired results.

Now it remains to verify \eqref{2.5}. By the definition of $(\underline{u},\ud v, \underline{h})$, the equalities in the fifth line of \eqref{2.5} hold.
Now we prove the inequalities in the first two lines of \eqref{2.5}.
 For $x\in[\ud h(t)/4,\ud h(t)]$, if $l_1\theta^{\frac{1}{\gamma-1}}\ln^{\frac{\beta}{\gamma-1}}\theta$ is suitably large, with the aid of \eqref{1.7} and
 \[\hat{\zeta}(t,x)\ge\frac{\ud h(t)-x}{\ud h(t)},\]
 we have
\bes\label{2.10}
&&\int_{X_{\ep}}^{\underline h(t)}J_i(x-y)\hat\zeta(t,y){\rm d}y\nonumber\\
&=&\int_{X_{\ep}-x}^{\underline h(t)-x}J_i(y)\hat\zeta(t,x+y){\rm d}y\nonumber\\
  &\ge&\int_{X_{\ep}-\frac{\underline{h}(t)}{4}}^{-\frac{\underline{h}(t)}{8}}J_i(y)\hat\zeta(t,x+y){\rm d}y\nonumber\\
  &\ge& \varsigma_1\int_{-\frac{\underline{h}(t)}{6}}^{-\frac{\underline{h}(t)}{8}}
 \frac{\ln^{\beta}|y|}{|y|^{\gamma}}\frac{\underline{h}(t)-x-y}{\underline{h}(t)}{\rm d}y\nonumber\\
&\ge&\frac{\varsigma_1}{\underline{h}(t)}\int_{-\frac{\underline{h}(t)}{6}}^{-\frac{\underline{h}(t)}{8}}
\frac{\ln^{\beta}|y|}{|y|^{\gamma}}(-y){\rm d}y\nonumber\\
&=&\frac{\varsigma_1}{\underline{h}(t)}\int_{\frac{\underline{h}(t)}{8}}^{\frac{\underline{h}(t)}{6}}
 y^{1-\gamma}\ln^{\beta}y{\rm d}y\nonumber\\
 &=&\frac{\varsigma_1}{\underline{h}(t)}\left[\frac{y^{2-\gamma}}{2-\gamma}\ln^{\beta}y\bigg|^{\frac{\ud h(t)}{6}}_{\frac{\ud h(t)}{8}}-\frac{\beta}{2-\gamma}\int_{\frac{\ud h(t)}{8}}^{\frac{\ud h(t)}{6}}y^{1-\gamma}\ln^{\beta-1}y{\rm d}y\right]\nonumber\\
 &=&\frac{\varsigma_1\ud h^{1-\gamma}(t)}{2-\gamma}\left[\frac{(\ln h(t)-\ln 6)^{\beta}}{6^{2-\gamma}}-\frac{(\ln h(t)-\ln 8)^{\beta}}{8^{2-\gamma}}\right]-\frac{\varsigma_1\beta}{(2-\gamma)\ud h(t)}\int_{\frac{\ud h(t)}{8}}^{\frac{\ud h(t)}{6}}y^{1-\gamma}\ln^{\beta-1}y{\rm d}y\nonumber\\
 &\ge&\frac{\varsigma_1\ud h^{1-\gamma}(t)\ln^{\beta}\ud h(t)}{2(2-\gamma)}(\frac{1}{6^{2-\gamma}}-\frac{1}{8^{2-\gamma}})-\frac{\varsigma_1\beta}{(2-\gamma)\ud h(t)}\int_{\frac{\ud h(t)}{8}}^{\frac{\ud h(t)}{6}}y^{1-\gamma}\ln^{\beta-1}y{\rm d}y\nonumber\\
 &=&\ud h^{1-\gamma}(t)\ln^{\beta}\ud h(t)\left[\frac{\varsigma_1}{2(2-\gamma)}(\frac{1}{6^{2-\gamma}}-\frac{1}{8^{2-\gamma}})+o_{\ud h\to\yy}(1)\right]\nonumber\\
&\ge& \frac{\varsigma_1}{4(2-\gamma)}(\frac{1}{6^{2-\gamma}}-\frac{1}{8^{2-\gamma}})\ud h^{1-\gamma}(t)\ln^{\beta}\ud h(t)\nonumber\\
&=&\frac{\varsigma_1\ep}{4(2-\gamma)}(\frac{1}{6^{2-\gamma}}-\frac{1}{8^{2-\gamma}})l^{1-\gamma}_1(t+\theta)^{-1}\ln^{-\beta}(t+\theta)\ln^{\beta}[l_1(t+\theta)^{\frac{1}{\gamma-1}}\ln^{\frac{\beta}{\gamma-1}}(t+\theta)]\nonumber\\
&=&\frac{\varsigma_1\ep l^{1-\gamma}_1}{4(2-\gamma)}(\frac{1}{6^{2-\gamma}}-\frac{1}{8^{2-\gamma}})\frac{(t+\theta)^{-1}}{\ln^{\beta}(t+\theta)}\left[\ln l_1+\frac{1}{\gamma-1}\ln(t+\theta)+\frac{\beta}{\gamma-1}\ln\ln(t+\theta)\right]^{\beta}\nonumber\\
&\ge&\frac{\varsigma_1\ep}{4(2-\gamma)(\gamma-1)^{\beta}}(\frac{1}{6^{2-\gamma}}-\frac{1}{8^{2-\gamma}})l^{1-\gamma}_1\min\{\frac{1}{2^{\beta}},1\}(t+\theta)^{-1},
\ees
where $l_1\theta^{\frac{1}{\gamma-1}}\ln^{\frac{\beta}{\gamma-1}}\theta$ is large enough such that
\[\left[\frac{(\ln h(t)-\ln 6)^{\beta}}{6^{2-\gamma}}-\frac{(\ln h(t)-\ln 8)^{\beta}}{8^{2-\gamma}}\right]\ge\frac{\ln^{\beta}\ud h(t)}{2}(\frac{1}{6^{2-\gamma}}-\frac{1}{8^{2-\gamma}}).\]
Clearly, using \eqref{2.9} and arguing as in the proof of Lemma \ref{l2.2}, we have that if $\ep$ is suitably small, for $x\in[k_0,\underline{h}(t)]$, \eqref{2.8} is also valid here.

On the other hand, we have $\underline{u}_t(t,x)=0$ for $t>0$ and $x\in[0,{\underline{h}(t)}/2]$. Combining with \eqref{2.8}, the first inequality of \eqref{2.8} holds for $x\in[k_0,\ud h(t)/2)$. For $t>0$ and $x\in(\underline{h}(t)/2,\underline{h}(t))$, due to \eqref{2.10}, we have
\bess
\underline{u}_t(t,x)&=&2 u^*_{\ep}\frac{x\underline{h}'(t)}{\underline{h}^2(t)}\le 2 u^*_{\ep}\frac{\underline{h}'(t)}{\underline{h}(t)}\\
&=&\frac{2u^*_{\ep}(t+\theta)^{-1}}{\gamma-1}\left[1+\beta\ln^{-1}(t+\theta)\right]\\
&\le&\frac{4u^*_{\ep}(t+\theta)^{-1}}{\gamma-1}\\
&\le&\frac{\varsigma_1\tilde{r}\ep}{8(2-\gamma)(\gamma-1)^{\beta}}(\frac{1}{6^{2-\gamma}}-\frac{1}{8^{2-\gamma}})l^{1-\gamma}_1\min\{\frac{1}{2^{\beta}},1\}(t+\theta)^{-1}
\eess
provided that $l_1$ is small enough such that
\[\frac{4u^*_{\ep}}{\gamma-1}\le\frac{\varsigma_1\tilde{r}\ep}{8(2-\gamma)(\gamma-1)^{\beta}}(\frac{1}{6^{2-\gamma}}-\frac{1}{8^{2-\gamma}})l^{1-\gamma}_1\min\{\frac{1}{2^{\beta}},1\}.\]
So the inequality in the first line of \eqref{2.5} holds. Analogously, we can show the inequality in the third line of \eqref{2.5} is valid if $\ep$, $l_1$ and $\theta$ are chosen as above.

Then we prove the inequality in the fourth line of \eqref{2.5} holds. Straightforward computations show that if $l_1\theta^{\frac{1}{\gamma-1}}\ln^{\frac{\beta}{\gamma-1}}\theta$ is suitably large, then
\bess
&&\displaystyle\int_{X_{\ep}}^{\underline h(t)}\int_{\underline h(t)}^{\infty}
\big(\mu_1J_1(x-y)\underline u(t,x)+\mu_2J_2(x-y)\ud v(t,x)\big){\rm d}y{\rm d}x\\
&&\ge\int_{\frac{\ud h(t)}{2}}^{\underline h(t)}\int_{\underline h(t)}^{\infty}
2\big(\mu_1u^*_{\ep}J_1(x-y)+\mu_2v^*_{\ep}J_2(x-y)\big)\frac{\ud h(t)-x}{\ud h(t)}{\rm d}y{\rm d}x\\
&&=\frac{2}{\ud h(t)}\int_{-\frac{\ud h(t)}{2}}^{0}\int_{0}^{\infty}
\big(\mu_1u^*_{\ep}J_1(x-y)+\mu_2v^*_{\ep}J_2(x-y)\big)(-x){\rm d}y{\rm d}x\\
&&=\frac{2}{\ud h(t)}\left\{\int_{0}^{\frac{\ud h(t)}{2}}\int_{0}^{y}+\int_{\frac{\ud h(t)}{2}}^{\yy}\int_{0}^{\frac{\ud h(t)}{2}}\right\}
\big(\mu_1u^*_{\ep}J_1(y)+\mu_2v^*_{\ep}J_2(y)\big)x{\rm d}x{\rm d}y\\
&&\ge\frac{1}{\ud h(t)}\int_{\frac{\ud h(t)}{4}}^{\frac{\ud h(t)}{2}}\big(\mu_1u^*_{\ep}J_1(y)+\mu_2v^*_{\ep}J_2(y)\big)y^2{\rm d}y\\
&&\ge\frac{C_1}{\ud h(t)}\int_{\frac{\ud h(t)}{4}}^{\frac{\ud h(t)}{2}}y^{2-\gamma}\ln^{\beta}y{\rm d}y\\
&&=\frac{C_1}{\ud h(t)}\left[\ln^{\beta}y\frac{y^{3-\gamma}}{3-\gamma}\bigg|^{\frac{\ud h(t)}{2}}_{\frac{\ud h(t)}{4}}-\frac{\beta}{3-\gamma}\int_{\frac{\ud h(t)}{4}}^{\frac{\ud h(t)}{2}}y^{2-\gamma}\ln^{\beta-1}y\dy\right]\\
&&=\frac{C_1\ud h^{2-\gamma}(t)}{3-\gamma}\left[\frac{(\ln \ud h(t)-\ln 2)^{\beta}}{2^{3-\gamma}}-\frac{(\ln \ud h(t)-\ln 4)^{\beta}}{4^{3-\gamma}}\right]-\frac{C_1\beta}{(3-\gamma)\ud h(t)}\int_{\frac{\ud h(t)}{4}}^{\frac{\ud h(t)}{2}}y^{2-\gamma}\ln^{\beta-1}y\dy\\
&&\ge\frac{C_1\ud h^{2-\gamma}(t)\ln^{\beta}\ud h(t)}{2(3-\gamma)}(\frac{1}{2^{3-\gamma}}-\frac{1}{4^{3-\gamma}})-\frac{C_1\beta}{(3-\gamma)\ud h(t)}\int_{\frac{\ud h(t)}{4}}^{\frac{\ud h(t)}{2}}y^{2-\gamma}\ln^{\beta-1}y\dy\\
&&\ge\ud h^{2-\gamma}(t)\ln^{\beta}\ud h(t)\left[\frac{C_1}{2(3-\gamma)}(\frac{1}{2^{3-\gamma}}-\frac{1}{4^{3-\gamma}})+o_{\ud h\to\yy}(1)\right]\\
&&\ge \frac{C_1}{4(3-\gamma)}(\frac{1}{2^{3-\gamma}}-\frac{1}{4^{3-\gamma}}) \ud h^{2-\gamma}(t)\ln^{\beta}\ud h(t)\\
&&=\frac{C_1}{4(3-\gamma)}(\frac{1}{2^{3-\gamma}}-\frac{1}{4^{3-\gamma}})l^{2-\gamma}_1(t+\theta)^{\frac{2-\gamma}{\gamma-1}}\ln^{\frac{\beta(2-\gamma)}{\gamma-1}}(t+\theta)\left[\ln l_1+\frac{1}{\gamma-1}(\ln(t+\theta)+\beta\ln\ln(t+\theta))\right]^{\beta}\\
&&\ge \frac{C_1}{4(3-\gamma)}(\frac{1}{2^{3-\gamma}}-\frac{1}{4^{3-\gamma}})l^{2-\gamma}_1(t+\theta)^{\frac{2-\gamma}{\gamma-1}}\ln^{\frac{\beta(2-\gamma)}{\gamma-1}}(t+\theta)\min\{\frac{1}{2^{\beta}},1\}\frac{1}{(\gamma-1)^{\beta}}\ln^{\beta}(t+\theta)\\
&&=\min\{\frac{1}{2^{\beta}},1\}\frac{C_1l^{2-\gamma}_1}{4(\gamma-1)^{\beta}(3-\gamma)}(\frac{1}{2^{3-\gamma}}-\frac{1}{4^{3-\gamma}})(t+\theta)^{\frac{2-\gamma}{\gamma-1}}\ln^{\frac{\beta}{\gamma-1}}(t+\theta),
\eess
where $C_1$ depends only on $a$, $b$, $H$, $G$, $\mu_i$ and $J_i$ with $i=1,2$, and $l_1\theta^{\frac{1}{\gamma-1}}\ln^{\frac{\beta}{\gamma-1}}\theta$ is sufficiently large such that
\[\left[\frac{(\ln \ud h(t)-\ln 2)^{\beta}}{2^{3-\gamma}}-\frac{(\ln \ud h(t)-\ln 4)^{\beta}}{4^{3-\gamma}}\right]\ge\frac{\ln^{\beta}\ud h(t)}{2}(\frac{1}{2^{3-\gamma}}-\frac{1}{4^{3-\gamma}}).\]
 On the other hand, simple computations arrive at
\bess
\underline{h}'(t)&=&\frac{l_1}{\gamma-1}(t+\theta)^{\frac{2-\gamma}{\gamma-1}}\ln^{\frac{\beta}{\gamma-1}}(t+\theta)[1+\beta\ln^{-1}(t+\theta)]\\
&\le&\frac{2l_1}{\gamma-1}(t+\theta)^{\frac{2-\gamma}{\gamma-1}}\ln^{\frac{\beta}{\gamma-1}}(t+\theta)\\
&\le&\min\{\frac{1}{2^{\beta}},1\}\frac{C_1l^{2-\gamma}_1}{4(\gamma-1)^{\beta}(3-\gamma)}(\frac{1}{2^{3-\gamma}}-\frac{1}{4^{3-\gamma}})(t+\theta)^{\frac{2-\gamma}{\gamma-1}}\ln^{\frac{\beta}{\gamma-1}}(t+\theta)
\eess
provided that $l_1$ is suitably small such that
\[\frac{2}{\gamma-1}\le\min\{\frac{1}{2^{\beta}},1\}\frac{C_1l^{1-\gamma}_1}{4(\gamma-1)^{\beta}(3-\gamma)}(\frac{1}{2^{3-\gamma}}-\frac{1}{4^{3-\gamma}}).\]
Thus the inequality in the fourth line of \eqref{2.5} holds.

Now to prove \eqref{2.5}, it remains to show the inequalities in the last two lines of \eqref{2.5}. Since one can do this by following the same lines as in the proof of Lemma \ref{l2.2}, we omit the details. Hence, \eqref{2.5} holds. The proof is finished.
 \end{proof}

We then consider the case where condition \eqref{1.8} holds with $\beta>-1$.

  \begin{lemma}{\label{l2.4}}Let $(u,v,h)$ be the unique solution of \eqref{1.1} and the condition \eqref{1.8} hold with $\beta>-1$. Then the following statements are valid.
 \begin{enumerate}[$(1)$]
 \item There exists a positive constant $C$ such that $h(t)\ge C t\ln^{\beta+1} t$ for all large $t$.
 \item For any $s(t)=o(t\ln^{\beta+1} t)$, we have
 \[\liminf_{t\to\yy}(u(t,x),v(t,x))\ge (U(x),V(x)) ~ ~ {\rm uniformly ~ in ~ }[0,s(t)].\]

 \end{enumerate}
 \end{lemma}

\begin{proof} Let $X_{\ep}$ and $k_0$ be defined as in the proof of Lemma \ref{l2.3} for any small $\ep>0$. Define
\[\underline{h}(t)=l_1(t+\theta)\ln^{\beta+1} (t+\theta), ~ ~ \underline{u}(t,x)=u^*_{\ep}\tilde{\zeta}(t,x),~ ~ \ud v(t,x)=v^*_{\ep}\tilde{\zeta}(t,x),\]
where positive constants $l_1\ll1$ and $\theta\gg1$, and
\bess
&\alpha\in(0,1), ~ \theta^{\alpha}>k_0, ~ l_1\theta^{1-\alpha}\ln^{\beta+1}\theta\ge 3\alpha, ~ l_1\theta\ln^{\beta+1}\theta-3\theta^{\alpha}>2k_0,\\
&(u^*_{\ep},v^*_{\ep})=(1-\sqrt{\ep})(u^*,v^*), ~ ~ \tilde{\zeta}(t,x)=\min\kk\{1,\,\frac{\underline{h}(t)-x}
{(t+\theta)^{\alpha}}\rr\} ~ ~ {\rm for ~ }t\ge0, ~ x\in[X_{\ep},\underline{h}(t)].\eess

Now we are in the position to show that by choosing suitable $(l_1,\theta,\ep,T)$, there holds:
\bes\label{2.11}
\left\{\begin{aligned}
&\underline u_t\le d_1\displaystyle\int_{X_{\ep}}^{\underline h(t)}J_1(x-y)\underline u(t,y){\rm d}y-d_1\underline u-a\underline u+H(\ud v),\\
&\hspace{50mm}t>0,~x\in[k_0,\underline h(t))\setminus\left\{\underline{h}(t)-(t+\theta)^{\alpha}\right\},\\
&\underline v_t\le d_2\displaystyle\int_{X_{\ep}}^{\underline h(t)}J_2(x-y)\underline v(t,y){\rm d}y-d_2\underline v-b\underline v+G(\ud u),\\
&\hspace{50mm} t>0,~x\in[k_0,\underline h(t))\setminus\left\{\underline{h}(t)-(t+\theta)^{\alpha}\right\},\\
&\underline u(t,\underline h(t))=\ud v(t,\ud h(t))=0,&& t>0,\\
&\underline h'(t)<\displaystyle\int_{X_{\ep}}^{\underline h(t)}\int_{\underline h(t)}^{\infty}
\big(\mu_1J_1(x-y)\underline u(t,x)+\mu_2J_2(x-y)\ud v(t,x)\big){\rm d}y{\rm d}x,&& t>0,\\
&\underline{u}(t,x)\le u(t+T,x), ~ \ud v(t,x)\le v(t+T,x), && t>0, ~ x\in[X_{\ep},k_0],\\
&\underline h(0)<h(T),\;\;\underline u(0,x)\le u(T,x),~ \ud v(0,x)\le v(T,x), && x\in[X_{\ep},\underline h(0)].
\end{aligned}\right.
\ees
Once it is done, by the definition of $(\underline{u},\ud v, \underline{h})$ and the similar arguments as in the proof of Lemma \ref{l2.2}, one readily derives the results as wanted.

Now let us prove \eqref{2.11}. The identities in the fifth line of \eqref{2.11} is obvious. We next deal with the inequality in the first line of \eqref{2.11}.
Clearly,
 \bess
 \underline{u}(t,x)\ge u^*_{\ep}\frac{\underline{h}-x}{2(t+\theta)^{\alpha}}\;\;\;{\rm for} \;\;x\in[\underline{h}(t)-2(t+\theta)^{\alpha},\underline{h}(t)].
  \eess

Thus, for $x\in[\underline{h}(t)-(t+\theta)^{\alpha},\underline{h}(t)]$, letting $\theta$ be sufficiently large, we have
 \bes\label{2.12}
\int_{X_{\ep}}^{\underline h(t)}J_1(x-y)\underline u(t,y){\rm d}y&=&\int_{X_{\ep}-x}^{\underline h(t)-x}J_1(y)\underline u(t,x+y){\rm d}y\nonumber\\
 &\ge&\frac{ u^*_{\ep}}2\int_{-(t+\theta)^{\alpha}}^{-(t+\theta)^{\alpha/2}}J_1(y)
 \frac{\underline{h}(t)-x-y}{(t+\theta)^{\alpha}}{\rm d}y\nonumber\\
&\ge&\frac{ u^*_{\ep}}2\int_{-(t+\theta)^{\alpha}}^{-(t+\theta)^{\alpha/2}}J_1(y)
\frac{-y}{(t+\theta)^{\alpha}}{\rm d}y\nonumber\\
 &\ge&\frac{ u^*_{\ep}\varsigma_1}2\int_{-(t+\theta)^{\alpha}}^{-(t+\theta)^{\alpha/2}}\frac{\ln^{\beta}|y|}{y^2}
 \frac{-y}{(t+\theta)^{\alpha}}{\rm d}y\nonumber\\
 &=&\frac{ u^*_{\ep}\varsigma_1}{2 (t+\theta)^{\alpha}}\int_{(t+\theta)^{\alpha/2}}^{(t+\theta)^{\alpha}}\frac{\ln^{\beta}y}{y}{\rm d}y\nonumber\\
 &=&\frac{ u^*_{\ep}\varsigma_1\ln^{\beta+1}(t+\theta)}{2 (t+\theta)^{\alpha}(\beta+1)}[\alpha^{\beta+1}-(\frac{\alpha}{2})^{\beta+1}].
\ees

Besides, by Proposition \ref{p2.1} and our choices of $l_1$ and $\theta$, we have
  \bes\label{2.13}
\int_{X_{\ep}}^{\underline{h}(t)}J_i(x-y)\tilde{\zeta}(t,y)\dy\ge(1-\ep^2)\tilde{\zeta}(t,x) ~ ~ {\rm for ~ }t\ge0,~ x\in[k_0,\underline{h}(t)], ~ i=1,2.
\ees
Certainly, using \eqref{2.13}, if $\ep$ is small enough, \eqref{2.8} is valid here.
On the other hand, we have $\underline{u}_t=0$ for $t>0$ and $x\in[k_0,\underline{h}(t)-(t+\theta)^{\alpha})$. Thus the first inequality of \eqref{2.11} holds for $x\in[k_0,\underline{h}(t)-(t+\theta)^{\alpha})$. Using \eqref{2.12} yields
 \bess
 \underline{u}_t&=&u^*_{\ep}\frac{\ud h'(t)(t+\theta)^{\alpha}-(\ud h(t)-x)\alpha(t+\theta)^{\alpha-1}}{(t+\theta)^{2\alpha}}\\
&\le&u^*_{\ep}\left[\frac{\ud h'(t)}{(t+\theta)^{\alpha}}+\frac{\alpha\ud h(t)}{(t+\theta)^{\alpha+1}}\right]\\
&=&\frac{l_1u^*_{\ep}\ln^{\beta+1}(t+\theta)}{(t+\theta)^{\alpha}}\left[1+\alpha+(\beta+1)\ln^{-1}(t+\theta)\right]\\
&\le&\frac{l_1u^*_{\ep}\ln^{\beta+1}(t+\theta)}{(t+\theta)^{\alpha}}(2+\alpha+\beta)\\
&\le&\frac{ u^*_{\ep}\varsigma_1\ln^{\beta+1}(t+\theta)}{2 (t+\theta)^{\alpha}(\beta+1)}[\alpha^{\beta+1}-(\frac{\alpha}{2})^{\beta+1}]
\eess
for $t>0$ and $x\in(\underline{h}(t)-(t+\theta)^{\alpha},\, \underline{h}(t)]$ provided that $l_1$ is small enough such that
\[l_1u^*_{\ep}(2+\alpha+\beta)\le\frac{u^*_{\ep}\varsigma_1}{2(\beta+1)}[\alpha^{\beta+1}-(\frac{\alpha}{2})^{\beta+1}].\]
So the first inequality holds. Similarly, we can show the inequality for $\ud v$, i.e., the inequality in the third line of \eqref{2.11} holds if $l_1$, $\ep$ and $\theta$ are chosen as above.

We now turn to the inequality of $\ud h'$, i.e., the one in the sixth line of \eqref{2.11}. Direct calculations show that if $\theta$ is large enough, then
\bess
&&\int_{X_{\ep}}^{\underline h(t)}\int_{\underline h(t)}^{\infty}
\big(\mu_1J_1(x-y)\underline u(t,x)+\mu_2J_2(x-y)\ud v(t,x)\big){\rm d}y{\rm d}x\\
&\ge&\int_{\frac{\underline{h}(t)}2}^{\underline h(t)-(t+\theta)^{\alpha}}\int_{\underline h(t)}^{\infty}
\big(\mu_1u^*_{\ep}J_1(x-y)+\mu_2v^*_{\ep}J_2(x-y)\big){\rm d}y{\rm d}x\\
&=&\left\{\int_{(t+\theta)^{\alpha}}^{\frac{\underline{h}(t)}2}
\int_{(t+\theta)^{\alpha}}^{y}+\int_{\frac{\underline{h}(t)}2}^{\infty}
\int_{(t+\theta)^{\alpha}}^{\frac{\underline{h}(t)}2}\right\}
\big(\mu_1u^*_{\ep}J_1(y)+\mu_2v^*_{\ep}J_2(y)\big){\rm d}x{\rm d}y\\
&\ge&\int_{(t+\theta)^{\alpha}}^{\frac{\underline{h}(t)}2}
\int_{(t+\theta)^{\alpha}}^{y}\big(\mu_1u^*_{\ep}J_1(y)+\mu_2v^*_{\ep}J_2(y)\big){\rm d}x{\rm d}y\\
&\ge&\int_{2(t+\theta)^{\alpha}}^{\frac{\underline{h}(t)}2}\big(\mu_1u^*_{\ep}J_1(y)+\mu_2v^*_{\ep}J_2(y)\big)[y-(t+\theta)^{\alpha}]{\rm d}y\\
&\ge&\int_{2(t+\theta)^{\alpha}}^{\frac{\underline{h}(t)}2}\big(\mu_1u^*_{\ep}J_1(y)+\mu_2v^*_{\ep}J_2(y)\big)\frac{y}{2}{\rm d}y\\
&\ge&C_2\int_{2(t+\theta)^{\alpha}}^{\frac{\underline{h}(t)}2}\frac{\ln^{\beta}y}{y}\dy\\
&=&\frac{C_2}{\beta+1}\left[\ln^{\beta+1}\frac{\ud h(t)}{2}-\ln^{\beta+1}[2(t+\theta)^{\alpha}]\right]\\
&=&\frac{C_2}{\beta+1}\left\{\left[\ln l_1+\ln(t+\theta)+(\beta+1)\ln\ln(t+\theta)-\ln2\right]^{\beta+1}-\left[\ln2+\alpha\ln(t+\theta)\right]^{\beta+1}\right\}\\
&\ge&\frac{C_2(1-\alpha^{\beta+1})}{\beta+1}\ln^{\beta+1}(t+\theta),
\eess
where $C_2$ depends only on $(a,b,H,G,J_1,J_2,\mu_1,\mu_2)$.
 Moreover, it is easy to see that
\bess
\underline{h}'(t)&=& l_1\ln^{\beta+1}(t+\theta)+ l_1(\beta+1)\ln^{\beta}(t+\theta)\\
&=&  l_1\ln^{\beta+1}(t+\theta)\left[1+(\beta+1)\ln^{-1}(t+\theta)\right]\\
&\le& l_1(2+\beta)\ln^{\beta+1}(t+\theta)\\
&\le&\frac{C_2(1-\alpha^{\beta+1})}{\beta+1}\ln^{\beta+1}(t+\theta)
\eess
provided that $l_1$ is suitably small such that
\[l_1(2+\beta)\le\frac{C_2(1-\alpha^{\beta+1})}{\beta+1}.\]

Now to prove \eqref{2.11}, it remains to verify the inequalities in the last two lines of \eqref{2.11}. Since one can do this by arguing as in the proof of Lemma \ref{l2.2}, we omit the details here. Hence \eqref{2.11} holds. The proof is ended.
\end{proof}

Then we turn to the case where the assumption \eqref{1.8} hold with $\beta=-1$.

\begin{lemma}\label{l2.5}Let $(u,v,h)$ be the unique solution of \eqref{1.1} and the assumption \eqref{1.8} hold with $\beta=-1$. Then the following statements are valid.
 \begin{enumerate}[$(1)$]
 \item There exists a positive constant $C$ such that $h(t)\ge C t\ln\ln t$ for all large $t$.
 \item For any $s(t)=o(t\ln\ln t)$, we have
 \[\liminf_{t\to\yy}(u(t,x),v(t,x))\ge (U(x),V(x)) ~ ~ {\rm uniformly ~ in ~ }[0,s(t)].\]
 \end{enumerate}
 \end{lemma}
\begin{proof}
Let $X_{\ep}$ and $k_0$ be defined as in the proof of Lemma \ref{l2.3} for any small $\ep>0$. Define
\[\underline{h}(t)=l_1(t+\theta)\ln\ln(t+\theta), ~ ~ \underline{u}(t,x)=u^*_{\ep}\bar{\zeta}(t,x),~ ~ \ud v(t,x)=v^*_{\ep}\bar{\zeta}(t,x),\]
where $l_1\ll1$ and $\theta\gg1$, $(u^*_{\ep},v^*_{\ep})$ is the same with that of Lemma \ref{l2.2}, and
\bess
&\alpha\in(0,\frac{1}{2}), ~ e^{\ln^{\alpha}\theta}>k_0, ~ l_1e^{\ln^{1-\alpha}\theta-1}\ge3\alpha, ~ l_1\theta\ln\ln\theta-3e^{\ln^{\alpha}\theta}>2k_0,\\
& \hat{\zeta}(t,x)=\min\kk\{1,\,\frac{\underline{h}(t)-x}
{e^{\ln^{\alpha}(t+\theta)}}\rr\} ~ ~ {\rm for ~ }t\ge0, ~ x\in[X_{\ep},\underline{h}(t)].
\eess

Now we are going to show that by choosing suitable $(l_1,\theta,T)$, \eqref{2.11} holds but with
\[x\in[k_0,\underline h(t))\setminus\left\{\underline{h}(t)-(t+\theta)^{\alpha}\right\} ~ {\rm replaced ~ by ~ }x\in[k_0,\underline h(t))\setminus\left\{\underline{h}(t)-e^{\ln^{\alpha}(t+\theta)}\right\}\]
in the first two inequalities of \eqref{2.11}.
Once we have done that, as in the proof of Lemma \ref{l2.2},  the desired results follow.

Now let us verify \eqref{2.11}. The identities in the fifth line of \eqref{2.11} is obvious.
Due to the above choices of $(l_1,\theta)$ and Proposition \ref{p2.1}, we have
\bes\label{2.14}
\int_{X_{\ep}}^{\underline{h}(t)}J_i(x-y)\bar{\zeta}(t,y)\dy\ge(1-\ep^2)\bar{\zeta}(t,x) ~ ~ {\rm for ~ }t\ge0,~ x\in[k_0,\underline{h}(t)], ~ i=1,2.
\ees
 Clearly,
 \bess
 \underline{u}(t,x)\ge u^*_{\ep}\frac{\underline{h}-x}{2e^{\ln^{\alpha}(t+\theta)}}\;\;\;{\rm for} \;\;x\in[\underline{h}(t)-2e^{\ln^{\alpha}(t+\theta)},\underline{h}(t)].
  \eess
Thus, for $x\in[\underline{h}(t)-e^{\ln^{\alpha}(t+\theta)},\underline{h}(t)]$, with $\theta$ large enough, we have
 \bes\label{2.15}
\int_{X_{\ep}}^{\underline h(t)}J_1(x-y)\underline u(t,y){\rm d}y&=&\int_{X_{\ep}-x}^{\underline h(t)-x}J_1(y)\underline u(t,x+y){\rm d}y\nonumber\\
 &\ge&\frac{ u^*_{\ep}}2\int_{-e^{\ln^{\alpha}(t+\theta)}}^{-e^{\ln^{\frac{\alpha}{2}}(t+\theta)}}J_1(y)
 \frac{\underline{h}(t)-x-y}{e^{\ln^{\alpha}(t+\theta)}}{\rm d}y\nonumber\\
&\ge&\frac{ u^*_{\ep}}2\int_{-e^{\ln^{\alpha}(t+\theta)}}^{-e^{\ln^{\frac{\alpha}{2}}(t+\theta)}}J_1(y)
\frac{-y}{e^{\ln^{\alpha}(t+\theta)}}{\rm d}y\nonumber\\
 &\ge&\frac{ u^*_{\ep}\varsigma_1}{2e^{\ln^{\alpha}(t+\theta)}}\int_{-e^{\ln^{\alpha}(t+\theta)}}^{-e^{\ln^{\frac{\alpha}{2}}(t+\theta)}}
 \frac{1}{|y|\ln|y|}{\rm d}y\nonumber\\
 &=&\frac{ u^*_{\ep}\varsigma_1}{2e^{\ln^{\alpha}(t+\theta)}}\int_{e^{\ln^{\frac{\alpha}{2}}(t+\theta)}}^{e^{\ln^{\alpha}(t+\theta)}}
 \frac{1}{y\ln y}{\rm d}y\nonumber\\
 &=&\frac{ u^*_{\ep}\varsigma_1\alpha\ln\ln(t+\theta)}{4e^{\ln^{\alpha}(t+\theta)}}.
\ees

By virtue of \eqref{2.14}, if $\ep$ is small enough, \eqref{2.8} also holds.
On the other hand, we have $\underline{u}_t=0$ for $t>0$ and $x\in[k_0,\underline{h}(t)-e^{\ln^{\alpha}(t+\theta)})$. Thanks to \eqref{2.8}, the first inequality of \eqref{2.11} holds for $x\in[k_0,\underline{h}(t)-e^{\ln^{\alpha}(t+\theta)})$. For $x\in(\underline{h}(t)-e^{\ln^{\alpha}(t+\theta)},\ud h(t))$, using \eqref{2.8} and \eqref{2.15}, we have
 \bess\underline{u}_t&\le&\frac{\ud h'(t)}{e^{\ln^{\alpha}(t+\theta)}}\\
 &=&\frac{l_1\ln\ln(t+\theta)+l_1\ln^{-1}(t+\theta)}{e^{\ln^{\alpha}(t+\theta)}}\\
 &\le&\frac{2l_1\ln\ln(t+\theta)}{e^{\ln^{\alpha}(t+\theta)}}\\
 &\le&\frac{ \tilde{r}\ep u^*_{\ep}\varsigma_1\alpha\ln\ln(t+\theta)}{8e^{\ln^{\alpha}(t+\theta)}}
 \eess
for $t>0$ and $x\in(\underline{h}(t)-e^{\ln^{\alpha}(t+\theta)},\, \underline{h}(t)]$ provided that $16l_1\le \tilde{r}\ep u^*_{\ep}\varsigma_1\alpha$.
 So the first inequality holds. Similarly, we can show the inequality for $\ud v$ is true if $l_1$, $\theta$ and $\ep$ are chosen as above.

We next show the inequality of $\ud h'(t)$. Direct calculations show that if $\theta$ is large enough, then
\bess
&&\int_{X_{\ep}}^{\underline h(t)}\int_{\underline h(t)}^{\infty}
\big(\mu_1J_1(x-y)\underline u(t,x)+\mu_2J_2(x-y)\ud v(t,x)\big){\rm d}y{\rm d}x\\
&\ge&\int_{\frac{\underline{h}(t)}2}^{\underline h(t)-e^{\ln^{\alpha}(t+\theta)}}\int_{\underline h(t)}^{\infty}
\big(\mu_1u^*_{\ep}J_1(x-y)+\mu_2v^*_{\ep}J_2(x-y)\big){\rm d}y{\rm d}x\\
&=&\left\{\int_{e^{\ln^{\alpha}(t+\theta)}}^{\frac{\underline{h}(t)}2}
\int_{e^{\ln^{\alpha}(t+\theta)}}^{y}+\int_{\frac{\underline{h}(t)}2}^{\infty}
\int_{e^{\ln^{\alpha}(t+\theta)}}^{\frac{\underline{h}(t)}2}\right\}
\big(\mu_1u^*_{\ep}J_1(y)+\mu_2v^*_{\ep}J_2(y)\big){\rm d}x{\rm d}y\\
&\ge&\int_{e^{\ln^{\alpha}(t+\theta)}}^{\frac{\underline{h}(t)}2}
\int_{(t+\theta)^{\alpha}}^{y}\big(\mu_1u^*_{\ep}J_1(y)+\mu_2v^*_{\ep}J_2(y)\big){\rm d}x{\rm d}y\\
&\ge&\int_{2e^{\ln^{\alpha}(t+\theta)}}^{\frac{\underline{h}(t)}2}\big(\mu_1u^*_{\ep}J_1(y)+\mu_2v^*_{\ep}J_2(y)\big)[y-e^{\ln^{\alpha}(t+\theta)}]{\rm d}y\\
&\ge&C_1\int_{2e^{\ln^{\alpha}(t+\theta)}}^{\frac{\underline{h}(t)}2}\frac{1}{y\ln y}\dy\\
&=&C_1\left\{\ln\left[\ln l_1+\ln(t+\theta)+\ln\ln\ln(t+\theta)-\ln2\right]-\ln\left[\ln 2+\ln^{\alpha}(t+\theta)\right]\right\}\\
&\ge&C_1(1-2\alpha)\ln\ln(t+\theta),
\eess
where $C_1$ depends only on $(a,b,H,G,J_1,J_2,\mu_1,\mu_2)$.
 Moreover, it is easy to see that
\[\underline{h}'(t)= l_1\ln\ln(t+\theta)+ \frac{l_1}{\ln(t+\theta)}\le 2 l_1\ln\ln(t+\theta)<C_1(1-2\alpha)\ln\ln(t+\theta)\]
provided that $2l_1<C_1(1-2\alpha)$.
Hence the inequality of $\ud h'(t)$, i.e., the one in the six line of \eqref{2.11} holds.

Hence, to prove \eqref{2.11}, it remains to show the inequalities in the last two lines of \eqref{2.11}, which can be done by using the analogous arguments as in the proof of Lemma \ref{l2.2}. The details are omitted here. Therefore, \eqref{2.11} holds. The proof is ended.
\end{proof}

 Clearly, Theorem \ref{t1.1} follows from Lemmas \ref{l2.2}-\ref{l2.5} and the fact that
 \[\limsup_{t\to\yy}(u(t,x),v(t,x))\le (U(x),V(x)) ~ ~ {\rm uniformly ~ in ~ }[0,\yy)\]
 which has been proved in \cite[Lemma 3.4]{LL2}.

\end{document}